\documentclass[11pt]{amsart}

\usepackage[margin=1.1in]{geometry}
\usepackage{amsmath,amssymb,amsthm,mathtools}
\usepackage{enumitem}
\usepackage[colorlinks=true,linkcolor=blue,citecolor=blue,urlcolor=blue]{hyperref}

\numberwithin{equation}{section}

\newtheorem{theorem}{Theorem}[section]
\newtheorem{proposition}[theorem]{Proposition}
\newtheorem{lemma}[theorem]{Lemma}
\newtheorem{corollary}[theorem]{Corollary}
\theoremstyle{definition}
\newtheorem{definition}[theorem]{Definition}
\newtheorem{assumption}[theorem]{Assumption}

\theoremstyle{remark}
\newtheorem{remark}[theorem]{Remark}

\newcommand{\R}{\mathbb R}
\newcommand{\C}{\mathbb C}

\newcommand{\calA}{\mathcal A}

\newcommand{\calK}{\mathcal K}
\newcommand{\calP}{\mathcal P}
\newcommand{\calQ}{\mathcal Q}

\newcommand{\PSH}{\operatorname{PSH}}
\newcommand{\MA}{\operatorname{MA}}
\newcommand{\Conv}{\operatorname{Conv}}
\newcommand{\ddc}{dd^c}
\newcommand{\supp}{\operatorname{supp}}

\newcommand{\bOmega}{\overline{\Omega}}

\title[Jensen Deficits]{Jensen Deficits for Inhomogeneous \\ Monge--Amp\`ere Dirichlet Problems}

\author{Frank Wikström}
\address{Center for Mathematical Sciences, Lund University, Box 118, SE-221 00 Lund, Sweden}
\email{frank.wikstrom@math.lth.se}


\begin{document}

\begin{abstract}
We develop an inhomogeneous form of Edwards' Jensen-measure duality for
Perron envelopes constrained by Monge--Amp\`ere lower bounds.  The
admissible subsolution families are convex but not cones; nevertheless,
the dual measures remain the homogeneous Jensen measures, and the
right-hand side enters through a scalar Jensen deficit
\[
        B_{\calA}(x,\mu)
        =
        \inf_{u\in\calA}
        \left(\int_{\partial\Omega}u\,d\mu-u(x)\right).
\]
Under natural structural hypotheses we prove a boundary dual formula
\[
        \sup\{u(x):u\in\calA,\ u\leq\varphi\text{ on }E\}
        =
        \inf_{\mu\in J_x^\partial}
        \left(
        \int_{\partial\Omega}\varphi\,d\mu
        -
        B_{\calA}(x,\mu)
        \right).
\]
We apply the theorem to real Alexandrov subsolutions and to complex
Bedford--Taylor plurisubharmonic subsolutions with continuous density.
In one real dimension the deficit is the Green-potential correction; in
higher dimensions it has intrinsic stress and current interpretations.
On B-regular domains, a bounded Bedford--Taylor approximation theorem
identifies bounded and continuous competitors and yields a duality proof
of continuity for the corresponding Dirichlet solution.  Finally, for
smooth strictly elliptic solutions, optimal Jensen measures are the
harmonic measures of the linearized Monge--Amp\`ere operators,
equivalently the boundary derivatives of the nonlinear solution map.
\end{abstract}

\maketitle

\section{Introduction}

Edwards' theorem, introduced in the setting of Choquet boundary theory
\cite{Edwards:66}, is one of the basic duality principles behind Jensen
measure representations of Perron envelopes.  In a typical form, if
$X$ is compact, $\mathcal F$ is a cone of upper semicontinuous functions
on $X$ containing the constants, and $\varphi$ is lower semicontinuous,
then
\[
        \sup\{u(x):u\in\mathcal F,\ u\leq \varphi\}
        =
        \inf_{\mu\in J_x(\mathcal F)}
        \int_X \varphi\,d\mu,
\]
where
\[
        J_x(\mathcal F)
        =
        \left\{
        \mu\in\calP(X):
        u(x)\leq \int_X u\,d\mu
        \text{ for all }u\in\mathcal F
        \right\}.
\]
For subharmonic, plurisubharmonic, and convex functions, this gives a
measure-theoretic representation of homogeneous Perron envelopes and of
solutions to homogeneous Dirichlet problems.

The cone hypothesis is precisely what fails for inhomogeneous
Monge--Amp\`ere Dirichlet problems.  The motivating example is the real
Monge--Amp\`ere envelope
\[
        P_{\nu}\varphi(x)
        =
        \sup\left\{
        u(x):
        u\in C(\bOmega)\cap \Conv(\Omega),\
        \MA_{\R}(u)\geq \nu,\
        u|_{\partial\Omega}\leq \varphi
        \right\}.
\]
The admissible family
\[
        \calA_\nu
        =
        \{u\in C(\bOmega)\cap\Conv(\Omega):\MA_{\R}(u)\geq \nu\}
\]
is convex and invariant under addition of affine functions, but it is
not a cone: rescaling a subsolution changes the right-hand side.  Thus
the classical Jensen-measure theorem does not directly represent the
corresponding Perron envelope.

The main point of this paper is that the right-hand side should not be
absorbed by changing the Jensen measures.  The homogeneous Jensen
measures remain the correct dual objects, and the inhomogeneity enters
through a scalar deficit.  For an admissible class $\calA$ define
\[
        B_{\calA}(x,\mu)
        =
        \inf_{u\in\calA}
        \left(
        \int_{\partial\Omega}u\,d\mu-u(x)
        \right),
\]
the least Jensen gap imposed by membership in $\calA$.  In the boundary
form relevant below, the abstract separation theorem
(Theorem~\ref{thm:abstract-edwards}) gives
\[
        P_{\calA}\varphi(x)
        =
        \inf_{\mu\in J_x^\partial}
        \left(
        \int_{\partial\Omega}\varphi\,d\mu
        -
        B_{\calA}(x,\mu)
        \right).
\]
For a homogeneous cone the deficit vanishes and this reduces to the
usual Edwards representation.

We apply this theorem in two parallel Monge--Amp\`ere settings: real
Alexandrov subsolutions with $\MA_{\R}(u)\geq f\,dx$
(Theorem~\ref{thm:real}) and complex Bedford--Taylor
plurisubharmonic subsolutions with $(\ddc u)^n\geq f\,dV$
(Theorem~\ref{thm:psh}).  In both cases the structural hypotheses come
from concavity of the determinant root, with Dinew's mixed
Monge--Amp\`ere inequality \cite[Theorem~1.3]{Dinew:09} supplying the
complex Bedford--Taylor case.
When $f=0$, the formulas specialize to the classical homogeneous Jensen
representations.

For convex functions, the Jensen measures themselves have a particularly
transparent interpretation.  If $K\subset\R^n$ is compact and convex and
$\mathcal F$ is the cone of continuous convex functions on $K$, then
$J_x(\mathcal F)$ is exactly the set of probability measures on $K$ with
barycenter $x$: affine functions force the barycenter condition, and
Jensen's inequality gives the converse.  Thus, in the convex case,
Jensen measures are the familiar representing measures of Choquet theory
\cite{Alfsen:71,Phelps:01}.  The boundary Jensen measures used below are
the corresponding representing measures whose support is restricted to
$\partial\Omega$.

In pluripotential theory, Jensen measures have become a flexible
language for plurisubharmonic envelopes, boundary regularity, and
approximation.  Sibony's work on $B$-regular domains~\cite{Sibony:87},
Poletsky's disc and current methods \cite{Poletsky:91,Poletsky:93}, and
the envelope theory of L{\'a}russon--Sigurdsson~\cite{LarussonSigurdsson:98}
are central points of reference.  Jensen measures also give useful
descriptions of hyperconvexity, boundary values, and approximation of
plurisubharmonic functions; see for
instance~\cite{CarlehedCegrellWikstrom:99,DieuWikstrom:05,Wikstrom:01}.
Ransford's survey~\cite{Ransford:01} is a good entry point into this
circle of ideas.

Two related recent developments are worth noting.  Nilsson--Wikstr\"om
treat variants of Edwards' theorem for cones that need not contain all
constants~\cite{NW}.  In another direction, di Nezza and Rashkovskii
study rooftop envelopes and geodesics for plurisubharmonic
functions~\cite{diNezza}, proving Edwards-type representation theorems
for $\omega$-plurisubharmonic envelopes on K\"ahler manifolds.  Since
the $\omega$-plurisubharmonic functions do not form a cone, their dual
objects are pairs consisting of a measure and an additive constant
adapted to the K\"ahler form~$\omega$; here we instead keep the ordinary
Jensen measures and place the inhomogeneity in the scalar deficit.

The deficit also has more concrete interpretations.  In one real
dimension it is the Green potential of the right-hand side.  In higher
dimensions, stress and current formulations give intrinsic lower bounds:
their inequalities are proved in the smooth or approximable settings
indicated there, while the Euler--Lagrange identifications are formal.
When a smooth strictly elliptic solution is given, the final section
identifies the optimal dual Jensen measures with harmonic measures for
the linearized Monge--Amp\`ere operator, or equivalently with the
boundary derivative of the nonlinear solution map.

This point of view is complementary to the nonlinear potential theory of
Harvey and Lawson
\cite{HarveyLawsonDirichlet:09,HarveyLawsonRiemannian:11}, whose
Dirichlet duality gives a broad comparison and existence framework for
closed degenerate elliptic subequations.  Here the aim is narrower: an
Edwards-type representation of a Perron envelope in which the
inhomogeneous constraint appears through the scalar deficit
$B_{\calA}$.  The stress and current representations are also close in
spirit to Duval--Sibony's current-theoretic duality and to
Harvey--Lawson's positive-current/Jensen-measure duality
\cite{DuvalSibony:95,HarveyLawsonCurrents:09}; the additional feature is
the determinant cost attached to the Monge--Amp\`ere lower bound.

The concrete theorems are proved for continuous densities,
$\nu=f\,dx$ in the real case and $\nu=f\,dV$ in the complex case.  The
measure notation above indicates the guiding problem; singular
right-hand sides require additional approximation or finite-energy
hypotheses and are left for later work.  Likewise, the purpose is not to
replace the classical Alexandrov and Bedford--Taylor existence and
regularity theories
\cite{BedfordTaylor:76,Blocki:96,Gutierrez:01}, but to supply a
Jensen-measure dual formula for inhomogeneous envelopes.  The bounded
Bedford--Taylor approximation theorem is used to identify bounded and
continuous competitors on $B$-regular domains and gives a duality proof
of continuity for the corresponding Dirichlet solution.

The paper is organized as follows.  We first prove the abstract
separation theorem and record some elementary consequences.  A
one-dimensional model then identifies the deficit with the Green
potential of the right-hand side.  The real Alexandrov formula and the
complex plurisubharmonic formula are proved next for continuous
densities, followed in the complex case by the bounded Bedford--Taylor
approximation theorem.  The stress and current subsections give the
intrinsic interpretation of the deficit, and the final section treats
linearized harmonic measures.

\subsection*{Acknowledgments}
The author used ChatGPT~5.5 during the preparation of this paper for
mathematical brainstorming, proof auditing, and expository revision.
All arguments, references, and final text were reviewed and verified by
the author, who takes full responsibility for the content of the paper.

The author would also like to thank Mårten Nilsson for valuable comments on an
earlier draft.

\section{The Abstract Inhomogeneous Edwards Theorem}

The separation argument in this section does not use convexity of the
domain, or any specific Monge--Amp\`ere structure.  We formulate it for
a compact constraint set~$E$, a distinguished point~$x$ at which
competitors are evaluated, and a homogeneous convex cone $\mathcal F$ defining
the Jensen measures.  In the Dirichlet problem applications $E$ will usually be
$\partial\Omega$ and $x\in\Omega$.

All functions below are assumed to have a well-defined value at $x$ and
a bounded upper semicontinuous trace on $E$, denoted by $u_E$.  For a
function which is continuous up to $E$, this trace is just the usual
restriction and for plurisubharmonic applications it may be the upper
semicontinuous boundary regularization.  Given a
cone $\mathcal F$ of such functions, we define the corresponding Jensen
measures by
\[
        J_x^E(\mathcal F)
        =
        \left\{
        \mu\in\calP(E):
        v(x)\leq \int_E v_E\,d\mu
        \text{ for all }v\in\mathcal F
        \right\}.
\]
We shall use the standard Choquet-theoretic fact that if $h$ is bounded
upper semicontinuous on a compact Hausdorff space and $\lambda$ is a
positive finite Borel measure, then
\[
        \inf_{\substack{g\in C(E)\\ g\geq h}}
        \int_E g\,d\lambda
        =
        \int_E h\,d\lambda .
\]
Indeed, the continuous majorants of $h$ form a downward directed family
with pointwise infimum $h$, and monotone convergence for directed nets
gives the identity; see, for example, \cite[Chapter~1]{Phelps:01}.

\begin{definition}
Let $\calA$ be a nonempty family of functions with bounded upper
semicontinuous trace on $E$.  For $\varphi\in C(E)$ define
\[
        P_{\calA}\varphi(x)
        =
        \sup\{u(x):u\in\calA,\ u_E\leq \varphi\}.
\]
For a finite positive Borel measure $\mu$ on $E$, define
\[
        B_{\calA}(x,\mu)
        =
        \inf_{u\in\calA}
        \left(
        \int_E u_E\,d\mu-u(x)
        \right).
\]
\end{definition}

In order to prove our duality result, we need to make
the following assumptions on $\mathcal{A}$ and $\mathcal{F}$:

\begin{assumption}\label{ass:abstract}
The pair $(\calA,\mathcal F)$ together with the trace operation
$u\mapsto u_E$ satisfies:
\begin{enumerate}[label=\textup{(\roman*)}]
    \item $\calA$ is convex and
    \[
            (t u+(1-t)v)_E
            \leq
            t u_E+(1-t)v_E
    \]
    for $u,v\in\calA$ and $0\leq t\leq1$.
    \item $\calA$ is invariant under addition of constants, and
    $(u+a)_E=u_E+a$.
    \item If $u\in\calA$, $v\in\mathcal F$, and $t\geq0$, then
    $u+t v\in\calA$ and
    \[
            (u+t v)_E\leq u_E+t v_E.
    \]
\end{enumerate}
\end{assumption}

The convex set used in the separation argument is
\[
        \calK
        =
        \left\{
        \left(g,\ u(x)-s\right):
        u\in\calA,\ g\in C(E),\ g\geq u_E,\ s\geq 0
        \right\}
        \subset C(E)\times\R.
\]
The point $(\varphi,c)$ belongs to $\calK$ exactly when there is an
admissible $u$ with $u_E\leq\varphi$ and $u(x)\geq c$.  We shall
separate from the closure of $\calK$; no closedness hypothesis on
$\calK$ is needed.  The trace compatibility in Assumption~\ref{ass:abstract}
ensures that $\calK$ is convex.  Indeed, if
$g_i\geq (u_i)_E$, $s_i\geq0$, and $0\leq t\leq1$, then
$u=t u_1+(1-t)u_2$ belongs to $\calA$ and
\[
        u_E\leq t(u_1)_E+(1-t)(u_2)_E
        \leq t g_1+(1-t)g_2,
\]
while $t s_1+(1-t)s_2\geq0$.

\begin{lemma}[Boundary stability]\label{lem:boundary-stability}
Assume $\calA$ is invariant under addition of constants.  Then for all
$g_1,g_2\in C(E)$,
\[
        P_{\calA}g_1(x)
        \leq
        P_{\calA}g_2(x)
        +
        \|g_1-g_2\|_{L^\infty(E)}.
\]
Consequently,
\[
        |P_{\calA}g_1(x)-P_{\calA}g_2(x)|
        \leq
        \|g_1-g_2\|_{L^\infty(E)}.
\]
\end{lemma}

\begin{proof}
Let $\varepsilon=\|g_1-g_2\|_{L^\infty(E)}$.  If
$u\in\calA$ and $u_E\leq g_1$, then $u-\varepsilon\in\calA$ and
\[
        (u-\varepsilon)_E\leq g_2.
\]
Thus
\[
        u(x)-\varepsilon\leq P_{\calA}g_2(x).
\]
Taking the supremum over such $u$ gives the first inequality.  The
second follows by symmetry.
\end{proof}

\begin{theorem}[Abstract inhomogeneous Edwards theorem]
\label{thm:abstract-edwards}
Let $E$ be compact Hausdorff, let $x$ be a distinguished point, and let
$\varphi\in C(E)$.  Assume that $(\calA,\mathcal F)$ satisfies
Assumption~\ref{ass:abstract} and that $P_{\calA}\varphi(x)$ is finite.
Then
\[
        P_{\calA}\varphi(x)
        =
        \inf_{\mu\in J_x^E(\mathcal F)}
        \left(
        \int_E\varphi\,d\mu
        -
        B_{\calA}(x,\mu)
        \right).
\]
\end{theorem}

\begin{proof}
We first prove weak duality.  Let $u\in\calA$ satisfy
$u_E\leq\varphi$, and let $\mu\in J_x^E(\mathcal F)$.  By the
definition of $B_{\calA}$,
\[
        B_{\calA}(x,\mu)
        \leq
        \int_E u_E\,d\mu-u(x).
\]
Therefore
\[
        u(x)
        \leq
        \int_E u_E\,d\mu-B_{\calA}(x,\mu)
        \leq
        \int_E\varphi\,d\mu-B_{\calA}(x,\mu).
\]
Taking the supremum over all admissible $u$ gives
\[
        P_{\calA}\varphi(x)
        \leq
        \int_E\varphi\,d\mu-B_{\calA}(x,\mu),
\]
and then taking the infimum over $\mu\in J_x^E(\mathcal F)$ gives one
inequality.

For the reverse inequality, fix $c>P_{\calA}\varphi(x)$.  We first show
that $(\varphi,c)$ does not belong to the closure of $\calK$.  Indeed,
if $(g_j,t_j)\in\calK$ and $(g_j,t_j)\to(\varphi,c)$ in
$C(E)\times\R$, then
\[
        t_j\leq P_{\calA}g_j(x).
\]
By Lemma~\ref{lem:boundary-stability},
\[
        P_{\calA}g_j(x)
        \leq
        P_{\calA}\varphi(x)
        +
        \|g_j-\varphi\|_{L^\infty(E)}.
\]
Letting $j\to\infty$ gives $c\leq P_{\calA}\varphi(x)$, a
contradiction.  Hence
\[
        (\varphi,c)\notin \overline{\calK}.
\]

Since $\overline{\calK}$ is closed and convex, the Hahn--Banach
separation theorem gives a nonzero continuous linear functional on
$C(E)\times\R$ and a strict separation.  Thus there exist a signed
finite Borel measure $\lambda$ on $E$ and a real number $\alpha$ such
that
\[
        \int_E\varphi\,d\lambda+\alpha c
        <
        \inf_{(g,t)\in\overline{\calK}}
        \left(
        \int_E g\,d\lambda+\alpha t
        \right).
\]
The infimum over $\overline{\calK}$ agrees with the infimum over
$\calK$ for this continuous linear functional.  Since $g\in C(E)$ may be
replaced by $g+h$ for arbitrary $h\in C(E)$, $h\geq0$, the infimum on
the right is finite only if $\lambda\geq0$.  Since $s\geq0$ is
arbitrary and $t=u(x)-s$, finiteness also forces $\alpha\leq0$.

We claim that $\alpha<0$.  If $\alpha=0$, then $\lambda$ is a nonzero
positive measure and
\[
        \int_E\varphi\,d\lambda
        <
        \inf_{\substack{u\in\calA\\ g\in C(E),\ g\geq u_E}}
        \int_E g\,d\lambda.
\]
Choose $u_0\in\calA$.  Since $u_{0,E}$ is bounded above and $\varphi$ is
continuous, subtracting a sufficiently large constant from $u_0$ gives a
new element, still denoted $u_0$, with $u_{0,E}<\varphi$ on $E$.  Since
$\lambda$ is nonzero and positive,
\[
        \int_E u_{0,E}\,d\lambda
        <
        \int_E\varphi\,d\lambda,
\]
and by the continuous-majorant formula above, the right-hand infimum is
at most $\int_E u_{0,E}\,d\lambda$.  This contradicts the previous
inequality.  Hence $\alpha<0$.

After multiplying the separating functional by a positive constant, we
may assume $\alpha=-1$.  The separation inequality becomes
\[
        \int_E\varphi\,d\lambda-c
        <
        \inf_{\substack{u\in\calA\\ g\in C(E),\ g\geq u_E}}
        \left(
        \int_E g\,d\lambda-u(x)
        \right).
\]
Since $\lambda\geq0$ and each trace $u_E$ is bounded upper
semicontinuous, the continuous-majorant formula gives
\[
        \inf_{\substack{g\in C(E)\\ g\geq u_E}}
        \int_E g\,d\lambda
        =
        \int_E u_E\,d\lambda.
\]
Thus
\[
        \int_E\varphi\,d\lambda-c
        <
        \inf_{u\in\calA}
        \left(
        \int_E u_E\,d\lambda-u(x)
        \right),
\]
and the infimum on the right is finite.

Adding constants to functions in $\calA$ identifies the total mass.  If
$a\in\R$ and $u\in\calA$, then $u+a\in\calA$ and
\[
        \int_E(u_E+a)\,d\lambda-(u(x)+a)
        =
        \int_E u_E\,d\lambda-u(x)
        +
        a(\lambda(E)-1).
\]
Since $a$ is arbitrary and the infimum is finite, we must have
$\lambda(E)=1$.

We now use the cone $\mathcal F$ to identify $\lambda$ as a Jensen
measure.  If $v\in\mathcal F$ and $t\geq0$, then $u+t v\in\calA$ and
\[
        \int_E(u+t v)_E\,d\lambda-(u(x)+t v(x))
        \leq
        \int_E u_E\,d\lambda-u(x)
        +
        t\left(
        \int_E v_E\,d\lambda-v(x)
        \right).
\]
If the coefficient of $t$ were negative, the infimum would be $-\infty$.
Therefore
\[
        v(x)\leq\int_E v_E\,d\lambda
        \qquad\text{for every }v\in\mathcal F.
\]
Together with $\lambda(E)=1$, this gives
$\lambda\in J_x^E(\mathcal F)$.  Therefore
\[
        \inf_{\mu\in J_x^E(\mathcal F)}
        \left(
        \int_E\varphi\,d\mu-B_{\calA}(x,\mu)
        \right)
        \leq
        \int_E\varphi\,d\lambda-B_{\calA}(x,\lambda)
        <
        c.
\]
Letting $c\downarrow P_{\calA}\varphi(x)$ gives the reverse inequality.
\end{proof}

\subsection*{The recession cone}

We remark that the homogeneous cone can be recovered from the inhomogeneous class.
If $\calA$ sits in an ambient class with the same trace operation, set
\[
        \operatorname{rec}\calA
        =
        \left\{
        v:\ u+t v\in\calA
        \text{ and }(u+t v)_E\leq u_E+t v_E
        \text{ for all }u\in\calA,\ t\geq0
        \right\}.
\]
This is the largest cone compatible with $\calA$ in the sense of
Assumption~\ref{ass:abstract}(iii).  Thus the Jensen class can be
generated from $\calA$ itself by taking
$\mathcal F=\operatorname{rec}\calA$.

In the Monge--Amp\`ere examples this recovers the usual homogeneous
cones.  For instance
$\operatorname{rec}\calA_f=C(\bOmega)\cap\Conv(\Omega)$ in the real
case: stability under adding convex functions gives one inclusion,
while if $v\in\operatorname{rec}\calA_f$ and $q$ is a strict quadratic
subsolution, then $q+t v$ is convex for all $t>0$, so
$v+t^{-1}q$ is convex and $v$ is convex by letting $t\to\infty$.
Similarly, the complex Bedford--Taylor class has recession cone
$\PSH(\Omega)\cap C(\bOmega)$.  This explains why the
Jensen measures are intrinsic to the inhomogeneous Monge--Amp\`ere
classes.

As in the homogeneous version of Edwards' theorem, the result
also holds for lower semicontinuous functions $\varphi$. More specifically,

\begin{corollary}[Lower semicontinuous data]\label{cor:lsc-edwards}
Let $E$ be compact Hausdorff, let $x$ be a distinguished point, and let
$\varphi:E\to(-\infty,+\infty]$ be lower semicontinuous and bounded from
below.  Assume that $(\calA,\mathcal F)$ satisfies
Assumption~\ref{ass:abstract} and that
$P_{\calA}\varphi(x)$ is finite, where
\[
        P_{\calA}\varphi(x)
        =
        \sup\{u(x):u\in\calA,\ u_E\leq \varphi\}.
\]
Then
\[
        P_{\calA}\varphi(x)
        =
        \inf_{\mu\in J_x^E(\mathcal F)}
        \left(
        \int_E\varphi\,d\mu
        -
        B_{\calA}(x,\mu)
        \right),
\]
with the integral understood in the extended sense.
\end{corollary}

\begin{proof}
Put $J=J_x^E(\mathcal F)$ and
\[
        \mathcal C_\varphi
        =
        \{\psi\in C(E):\psi\leq\varphi\}.
\]
This family is nonempty because $\varphi$ is bounded from below.  By
Kat\v{e}tov--Tong insertion,
\[
        P_{\calA}\varphi(x)
        =
        \sup_{\psi\in\mathcal C_\varphi}P_{\calA}\psi(x).
\]
Indeed, if $u\in\calA$ and $u_E\leq\varphi$, then some
$\psi\in C(E)$ satisfies
\[
        u_E\leq\psi\leq\varphi.
\]

For each $\psi\in\mathcal C_\varphi$, we have
$P_{\calA}\psi(x)\leq P_{\calA}\varphi(x)<+\infty$.  Hence, by
Theorem~\ref{thm:abstract-edwards},
\[
        P_{\calA}\psi(x)
        =
        \min_{\mu\in J}
        \left(
        \int_E\psi\,d\mu
        -
        B_{\calA}(x,\mu)
        \right).
\]
The minimum is attained since $J$ is compact and
$B_{\calA}(x,\cdot)$ is upper semicontinuous.

The functions
\[
        F_\psi(\mu)
        =
        \int_E\psi\,d\mu-B_{\calA}(x,\mu),
        \qquad
        \mu\in J,
\]
form a directed increasing family of lower semicontinuous functions on
$J$.  Hence the standard compactness lemma
\[
        \sup_{\psi\in\mathcal C_\varphi}\min_J F_\psi
        =
        \min_J\sup_{\psi\in\mathcal C_\varphi} F_\psi
\]
applies.  Finally, lower semicontinuous functions on compact Hausdorff
spaces are suprema of their continuous minorants, so for every
$\mu\in J$,
\[
        \sup_{\psi\in\mathcal C_\varphi}F_\psi(\mu)
        =
        \int_E\varphi\,d\mu-B_{\calA}(x,\mu).
\]
Combining these identities gives the desired formula.
\end{proof}

\begin{proposition}[Elementary consequences]\label{prop:abstract-consequences}
Let $\calA_1$ and $\calA_2$ be admissible families for the same
constraint set $E$, homogeneous class $\mathcal F$, trace operation, and
base point $x$, and suppose $\calA_2\subset\calA_1$.  Then, for every
$\mu\in J_x^E(\mathcal F)$,
\[
        B_{\calA_2}(x,\mu)\geq B_{\calA_1}(x,\mu).
\]
Consequently, whenever the deficit formula holds for both classes and
the same boundary data~$\varphi$,
\[
        P_{\calA_2}\varphi(x)\leq P_{\calA_1}\varphi(x).
\]

For a fixed admissible family $\calA$, if the deficit formula holds for
$\varphi,\psi\in C(E)$, then
\[
        |P_{\calA}\varphi(x)-P_{\calA}\psi(x)|
        \leq
        \|\varphi-\psi\|_{L^\infty(E)}.
\]
\end{proposition}

\begin{proof}
The first assertion follows directly from the definition of the deficit:
taking the infimum over the smaller family $\calA_2$ gives a larger
value.  The monotonicity of the envelopes follows by comparing the two
dual infima; the boundary term is the same and the larger deficit is
subtracted.

For the boundary-data estimate, every probability measure $\mu$ on $E$
satisfies
\[
        \int_E\varphi\,d\mu
        \leq
        \int_E\psi\,d\mu+\|\varphi-\psi\|_{L^\infty(E)}.
\]
The deficit term is the same for $\varphi$ and $\psi$.  Taking infima
over $\mu\in J_x^E(\mathcal F)$ gives one inequality, and the other
follows by symmetry.
\end{proof}

In the real and complex applications below, increasing the
right-hand side shrinks the admissible family, so the proposition gives
the usual monotonicity in the Monge--Amp\`ere data.

\begin{remark}[The convex boundary specialization]
\label{rem:convex-specialization}
Let $\Omega\subset\R^n$ be a bounded convex domain, let
$E=\partial\Omega$, and let $\mathcal F=C(\bOmega)\cap\Conv(\Omega)$.
Then
\[
        J_x^E(\mathcal F)
        =
        J_x^\partial
        :=
        \left\{
        \mu\in\calP(\partial\Omega):
        \int_{\partial\Omega} y\,d\mu(y)=x
        \right\}.
\]
Indeed, affine functions force the barycenter condition, and Jensen's
inequality gives the converse.  Thus the convexity of $\Omega$ enters
only in this specialization, not in Theorem~\ref{thm:abstract-edwards}
itself.

When $\calA$ is the full class of continuous convex functions on
$\bOmega$, affine functions show $B_{\calA}(x,\mu)=0$ for every
$\mu\in J_x^\partial$.  Theorem~\ref{thm:abstract-edwards} then reduces
to the usual convex Jensen representation of the convex envelope with
boundary data.
\end{remark}

\begin{remark}[Boundary support versus interior obstacles]
The restriction to boundary-supported measures is not essential for the
duality argument.  Theorem~\ref{thm:abstract-edwards} applies to any
compact constraint set $E$ once the corresponding Jensen class
$J_x^E(\mathcal F)$ is used.  For instance, in the convex setting one
may take $E=\bOmega$ and an obstacle $\Phi\in C(\bOmega)$, obtaining
\[
        P_{\calA}^{\bOmega}\Phi(x)
        =
        \inf_{\mu\in J_x(\bOmega)}
        \left(
        \int_{\bOmega}\Phi\,d\mu
        -
        B_{\calA}^{\bOmega}(x,\mu)
        \right),
\]
where $J_x(\bOmega)$ is the set of probability measures on $\bOmega$
with barycenter $x$ and $B_{\calA}^{\bOmega}$ is defined by integrating
over $\bOmega$.

Thus the boundary version should be viewed as the Dirichlet
specialization.  If one starts from an obstacle on $\bOmega$ and chooses
an arbitrary finite extension of boundary data, the resulting envelope
is an obstacle problem, not purely a Dirichlet problem.  This is visible
on the dual side: the full Jensen class contains interior measures, in
particular $\delta_x$, and for this measure the deficit is zero.  Hence
the dual formula sees the interior value $\Phi(x)$.  To recover the
boundary problem from the full-domain formula one may use the extended
obstacle $\Phi=+\infty$ in $\Omega$, or any finite extension known to
dominate all admissible competitors with the prescribed boundary data.
\end{remark}

\section{The One-Dimensional Model}

The deficit can be computed explicitly in one real dimension.  This is
the simplest indication that the inhomogeneous term should appear as a
Green-potential correction on the dual side.

\begin{proposition}[Green formula for the deficit]\label{prop:one-d}
Let $\Omega=(0,1)$, let $x\in(0,1)$, and let $\nu$ be a finite positive
Borel measure on $(0,1)$.  Consider
\[
        \calA_\nu
        =
        \{u\in C([0,1]): u \text{ is convex and } u''\geq\nu\}
\]
in the distributional sense.  Let
\[
        \mu_x=(1-x)\delta_0+x\delta_1.
\]
Then
\[
        B_\nu(x,\mu_x)
        =
        \int_0^1 G(x,t)\,d\nu(t),
\]
where
\[
        G(x,t)
        =
        \begin{cases}
        t(1-x), & t\leq x,\\
        x(1-t), & x\leq t.
        \end{cases}
\]
\end{proposition}

\begin{proof}
Let $L_u$ be the affine function interpolating the endpoint values of
$u$:
\[
        L_u(s)=(1-s)u(0)+s u(1).
\]
If $u''\geq\nu$, then the function
\[
        w(s)=L_u(s)-u(s)
\]
is nonnegative on $[0,1]$, vanishes at the endpoints, and satisfies
$-w''\geq\nu$ in the distributional sense.  The Green kernel for
$-d^2/ds^2$ on $(0,1)$ with zero boundary values is $G$.  Hence
\[
        L_u(x)-u(x)
        \geq
        \int_0^1 G(x,t)\,d\nu(t).
\]
Since
\[
        L_u(x)
        =
        \int_{\partial\Omega}u\,d\mu_x,
\]
we obtain
\[
        \int_{\partial\Omega}u\,d\mu_x-u(x)
        \geq
        \int_0^1 G(x,t)\,d\nu(t)
\]
for all $u\in\calA_\nu$. Equality is obtained by solving
\[
        -w''=\nu,\qquad w(0)=w(1)=0,
\]
and setting $u=L-w$ for an arbitrary affine function $L$.  Then
$u''=\nu$ and the displayed inequality is an equality.  Taking the
infimum over $u\in\calA_\nu$ proves the formula.
\end{proof}

\begin{corollary}
For $u\in\calA_\nu$,
\[
        u(x)
        \leq
        (1-x)u(0)+x u(1)
        -
        \int_0^1 G(x,t)\,d\nu(t).
\]
Thus the deficit is exactly the amount by which the usual linear
interpolation inequality is improved by the lower bound on $u''$.
\end{corollary}

\section{Real Alexandrov Monge--Amp\`ere Subsolutions}

We now turn from the one-dimensional model to convex Alexandrov
subsolutions.  Let $\Omega\subset\R^n$ be a bounded convex domain and
let $f\in C(\bOmega)$, $f\geq0$.  We write
\[
        \calA_f
        =
        \left\{
        u\in C(\bOmega)\cap\Conv(\Omega):
        \MA_{\R}(u)\geq f\,dx
        \right\},
\]
where $\MA_{\R}(u)$ denotes the Alexandrov Monge--Amp\`ere measure of
$u$.  The associated envelope is
\[
        P_f\varphi(x)
        =
        \sup\left\{
        u(x):
        u\in\calA_f,\
        u|_{\partial\Omega}\leq\varphi
        \right\},
\]
and the deficit is
\[
        B_f(x,\mu)
        =
        \inf_{u\in\calA_f}
        \left(
        \int_{\partial\Omega}u\,d\mu-u(x)
        \right).
\]

\begin{lemma}[Convexity and stability of $\calA_f$]\label{lem:real-structural}
The class $\calA_f$ is convex, invariant under addition of constants,
and stable under addition of nonnegative multiples of functions in
$C(\bOmega)\cap\Conv(\Omega)$.
\end{lemma}

\begin{proof}
Invariance under addition of constants is immediate.

We use the standard regularization and weak-continuity properties of the
Alexandrov Monge--Amp\`ere measure; see
\cite[Sections~1.1--1.2]{Gutierrez:01}.  In particular, if
$u\in\Conv(\Omega)$ and $\MA_{\R}(u)\geq f\,dx$, then on each
$U\Subset\Omega$ the mollifications $u_\varepsilon=u*\rho_\varepsilon$
satisfy
\[
        \det D^2u_\varepsilon
        \geq
        (f^{1/n}*\rho_\varepsilon)^n
        \quad\text{on }U,
\]
and $u_\varepsilon\to u$ locally uniformly, with Alexandrov measures
converging weakly.  The displayed estimate follows by writing the
absolutely continuous part of the Hessian measure as $A\,dx$ and using
Jensen's inequality for $\det^{1/n}$ on positive semidefinite matrices.

For smooth convex functions the assertions follow from the concavity and
monotonicity of $A\mapsto\det(A)^{1/n}$ on the positive semidefinite
cone.  Thus, if $u,v\in\calA_f$ and $0\leq t\leq1$, the preceding
estimate applied to $u_\varepsilon$ and $v_\varepsilon$ gives
\[
        \det D^2(tu_\varepsilon+(1-t)v_\varepsilon)
        \geq
        (f^{1/n}*\rho_\varepsilon)^n
        \quad\text{on }U.
\]
Letting $\varepsilon\to0$ and using weak continuity gives
$\MA_{\R}(tu+(1-t)v)\geq f\,dx$ on $U$, hence on $\Omega$.

The same argument proves stability under addition of nonnegative
multiples of functions in $C(\bOmega)\cap\Conv(\Omega)$: for
$w\in C(\bOmega)\cap\Conv(\Omega)$, $w_\varepsilon=w*\rho_\varepsilon$,
and $s\geq0$,
\[
        D^2(u_\varepsilon+s w_\varepsilon)\geq D^2u_\varepsilon,
\]
so determinant monotonicity gives the same lower bound, and the limit
again yields $\MA_{\R}(u+s w)\geq f\,dx$.
\end{proof}

\begin{lemma}[Strict subsolutions]\label{lem:strict-real}
If $f\in C(\bOmega)$ is bounded and nonnegative, then for every
$\varphi\in C(\partial\Omega)$ there exists $u_0\in\calA_f$ with
$u_0|_{\partial\Omega}<\varphi$.
\end{lemma}

\begin{proof}
Choose $M>0$ so large that $(2M)^n\geq \|f\|_{L^\infty(\Omega)}$ and set
\[
        q_M(y)=M|y|^2.
\]
Since $q_M$ is smooth and convex, its Alexandrov measure is
$\det D^2q_M\,dx=(2M)^n\,dx\geq f\,dx$.  For $C>0$ large enough,
$u_0=q_M-C$ satisfies $u_0<\varphi$ on $\partial\Omega$.
\end{proof}

\begin{remark}
The abstract theorem does not require the corresponding epigraph to be
closed.  This is important for the applications: the dual formula below
is obtained from convexity and stability under adding homogeneous
convex functions, rather than from the pre-existing Alexandrov
Dirichlet existence theorem.
\end{remark}

\begin{theorem}[Real inhomogeneous Jensen formula]\label{thm:real}
Let $\Omega\subset\R^n$ be a bounded convex domain, $x\in\Omega$, and
$f\in C(\bOmega)$, $f\geq0$.  Then for every
$\varphi\in C(\partial\Omega)$,
\[
        P_f\varphi(x)
        =
        \inf_{\mu\in J_x^\partial}
        \left(
        \int_{\partial\Omega}\varphi\,d\mu
        -
        B_f(x,\mu)
        \right).
\]
\end{theorem}

\begin{proof}
Apply Theorem~\ref{thm:abstract-edwards} with
$E=\partial\Omega$ and
$\mathcal F=C(\bOmega)\cap\Conv(\Omega)$.  By
Lemma~\ref{lem:real-structural}, the pair
$(\calA_f,\mathcal F)$ satisfies the convexity and stability
hypotheses.  Lemma~\ref{lem:strict-real} gives an admissible competitor
below the prescribed boundary data, so the envelope is not $-\infty$.
It is not $+\infty$ because convex functions with a uniform boundary
upper bound are bounded above at each interior point.  Thus the
finiteness hypothesis in Theorem~\ref{thm:abstract-edwards} is
satisfied.  Finally, Remark~\ref{rem:convex-specialization} identifies
$J_x^E(\mathcal F)$ with $J_x^\partial$.
\end{proof}

\subsection*{A stress representation of the real deficit}

There is a useful way to rewrite the real deficit which separates the
ordinary Jensen gap from the Monge--Amp\`ere constraint.  Fix
$x\in\Omega$ and $\mu\in J_x^\partial$.  A \emph{Jensen stress} for
$(x,\mu)$ is a positive semidefinite symmetric matrix-valued measure
$\Gamma$ on $\Omega$ such that
\begin{equation}\label{eq:jensen-stress}
        \int_{\partial\Omega}\phi\,d\mu-\phi(x)
        =
        \int_{\Omega}D^2\phi:d\Gamma
\end{equation}
for every $\phi\in C^2(\bOmega)$.  Equivalently,
$\operatorname{div}\operatorname{div}\Gamma=\mu-\delta_x$ in the weak
sense, with the boundary term encoded by the right-hand side of
\eqref{eq:jensen-stress}.  Here $A:B=\operatorname{tr}(AB)$ denotes the
usual contraction of symmetric matrices, so if
$\Gamma=(\Gamma_{ij})$ then
\[
        \int_{\Omega}D^2\phi:d\Gamma
        =
        \sum_{i,j}
        \int_{\Omega}
        \frac{\partial^2\phi}{\partial x_i\partial x_j}
        \,d\Gamma_{ij}.
\]

\begin{proposition}[Real Jensen stresses and lower bound]
\label{prop:real-stress-lower-bound}
Let $\Omega\subset\R^n$ be bounded and convex, let $x\in\Omega$, and
let $\mu\in J_x^\partial$.  Then Jensen stresses for $(x,\mu)$ exist.
Moreover, if $\Gamma=G\,dx+\Gamma_s$ is any Jensen stress and
$u\in C^2(\bOmega)$ is convex with $\det D^2u\geq f$, then
\begin{equation}\label{eq:stress-lower-bound}
        \int_{\partial\Omega}u\,d\mu-u(x)
        \geq
        n\int_{\Omega}f^{1/n}\det(G)^{1/n}\,dx .
\end{equation}
Consequently the same lower bound holds for the smooth deficit obtained
by taking the infimum over such smooth competitors.
\end{proposition}

\begin{proof}
Taylor's formula on the segment from $x$ to $y$ gives
\[
        \phi(y)-\phi(x)-D\phi(x)\cdot(y-x)
        =
        \int_0^1
        (1-t)(y-x)^T
        D^2\phi(x+t(y-x))
        (y-x)\,dt .
\]
After integration against $\mu$, the first order term vanishes because
$\mu$ has barycenter $x$.  Thus one obtains a stress by setting, for
continuous symmetric matrix fields $M$,
\[
        \int_{\Omega} M:d\Gamma_{x,\mu}
        =
        \int_{\partial\Omega}\int_0^1
        (1-t)(y-x)^T
        M(x+t(y-x))
        (y-x)\,dt\,d\mu(y).
\]
This is a higher-dimensional replacement for the Green kernel appearing
in the one-dimensional model above.

For the lower bound, let $u\in C^2(\bOmega)$ be convex with
$\det D^2u\geq f$.  By the stress identity,
\[
        \int_{\partial\Omega}u\,d\mu-u(x)
        =
        \int_{\Omega}D^2u:d\Gamma.
\]
If $\Gamma=G\,dx+\Gamma_s$ is the Lebesgue decomposition of the stress,
then positivity of $D^2u$ and $\Gamma_s$ gives
\[
        \int_{\partial\Omega}u\,d\mu-u(x)
        \geq
        \int_{\Omega}\operatorname{tr}(D^2u\,G)\,dx.
\]
and from the pointwise matrix arithmetic-geometric mean inequality it
follows that
\[
        \operatorname{tr}(D^2u\,G)
        \geq
        n\,\det(D^2u)^{1/n}\det(G)^{1/n}
        \geq
        n\,f^{1/n}\det(G)^{1/n}.
\]
For positive definite matrices this is the scalar AM--GM inequality
applied to the eigenvalues of
$G^{1/2}D^2u\,G^{1/2}$, and the semidefinite case follows by
approximation; see \cite[Chapter~IV]{Bhatia:97}.
Integrating gives \eqref{eq:stress-lower-bound}.
\end{proof}

The lower bound extends to Alexandrov competitors whenever the stress
pairing with $D^2u$ is justified by a smooth approximation preserving
the boundary pairing and the Monge--Amp\`ere lower bound.  In this
qualified smooth or approximable setting, every absolutely continuous
part $G\,dx$ of every Jensen stress gives a lower bound for $B_f(x,\mu)$.

This suggests the intrinsic relaxed deficit
\begin{equation}\label{eq:relaxed-deficit}
        \mathcal D_f(x,\mu)
        =
        \sup_{\Gamma}
        n\int_{\Omega}f^{1/n}\det(G_\Gamma)^{1/n}\,dx,
\end{equation}
where the supremum is over all Jensen stresses and
$G_\Gamma\,dx$ denotes the absolutely continuous part of $\Gamma$.
The preceding argument gives, at least in the smooth or approximable
setting,
\[
        \mathcal D_f(x,\mu)\leq B_f(x,\mu).
\]
Thus $B_f$ dominates a quantity depending only on the right-hand side
$f$ and on the second-order transport data carrying $\delta_x$ to
$\mu$.  Whether $\mathcal D_f(x,\mu)=B_f(x,\mu)$ in natural nonsmooth
classes is left open; such an identity would give a precise nonlinear
adjoint formula for the deficit.

\begin{remark}[Formal equality case]
The equality case in \eqref{eq:stress-lower-bound} suggests a more
intrinsic Euler--Lagrange picture.  If a smooth minimizer $u$ and a
maximizing stress $\Gamma=G\,dx$ exist and no relaxation gap occurs,
then equality in the matrix arithmetic-geometric mean would force
\[
        G=a\,\operatorname{cof}D^2u
\]
for some nonnegative scalar density $a$.  The stress equation becomes
\[
        \int_{\partial\Omega}\phi\,d\mu-\phi(x)
        =
        \int_{\Omega}
        a\,\operatorname{cof}D^2u:D^2\phi\,dx,
        \qquad \phi\in C^2(\bOmega),
\]
or, formally,
\[
        \operatorname{div}\operatorname{div}
        \bigl(a\,\operatorname{cof}D^2u\bigr)
        =
        \mu-\delta_x.
\]
Together with $\det D^2u=f$, this gives
\[
        B_f(x,\mu)
        =
        n\int_{\Omega}a\,f\,dx.
\]
Thus the deficit may be viewed as the potential of $f$ against the
nonnegative adjoint Green density for the linearized Monge--Amp\`ere
operator at the extremal convex function.  In dimension one the cofactor
is identically one and this reduces exactly to the usual Green function
formula.  The rigorous smooth version of this picture is realized in
Section~\ref{sec:linearized-harmonic-measure}, where the Green kernel of
the linearized operator supplies the density $a$.
\end{remark}

\section[Plurisubharmonic Subsolutions]{Plurisubharmonic Monge--Amp\`ere Subsolutions:\\ Continuous and Bounded Cases}

We now describe the parallel complex statement.  Here the ordinary Jensen
measures are no longer barycentric measures but instead the usual
boundary-supported plurisubharmonic Jensen measures. Let $\Omega\subset\C^n$ be
a bounded domain and $x\in\Omega$.  Define
\[
        J_{x,\PSH}^{\partial}
        =
        \left\{
        \mu\in\calP(\partial\Omega):
        v(x)\leq\int_{\partial\Omega}v\,d\mu
        \text{ for all }
        v\in \PSH(\Omega)\cap C(\bOmega)
        \right\}.
\]
For $f\in C(\bOmega)$, $f\geq0$, set
\[
        \calP_f
        =
        \left\{
        u\in \PSH(\Omega)\cap C(\bOmega):
        (\ddc u)^n\geq f\,dV
        \right\}.
\]
The envelope and deficit are
\[
        P_f^{\PSH}\varphi(x)
        =
        \sup\left\{
        u(x):
        u\in\calP_f,\
        u|_{\partial\Omega}\leq\varphi
        \right\},
\]
and
\[
        B_f^{\PSH}(x,\mu)
        =
        \inf_{u\in\calP_f}
        \left(
        \int_{\partial\Omega}u\,d\mu-u(x)
        \right).
\]

\begin{lemma}[Complex structural properties]\label{lem:complex-structural}
The class $\calP_f$ is convex, invariant under addition of constants,
and stable under addition of nonnegative multiples of functions in
$\PSH(\Omega)\cap C(\bOmega)$.
\end{lemma}

\begin{proof}
Invariance under addition of constants is immediate.  We use
Bedford--Taylor mixed products for continuous plurisubharmonic functions
\cite{BedfordTaylor:82} and Dinew's mixed Monge--Amp\`ere inequality
\cite[Theorem~1.3]{Dinew:09}: if bounded plurisubharmonic functions
$u_1,\ldots,u_n$ satisfy $(\ddc u_j)^n\geq f_j\,dV$, then
\[
        \ddc u_1\wedge\cdots\wedge\ddc u_n
        \geq
        (f_1\cdots f_n)^{1/n}\,dV .
\]
If $u,v\in\calP_f$ and $0\leq t\leq1$, multilinearity gives
\[
        \bigl(\ddc(tu+(1-t)v)\bigr)^n
        =
        \sum_{k=0}^n
        \binom{n}{k}
        t^k(1-t)^{n-k}
        (\ddc u)^k\wedge(\ddc v)^{n-k}.
\]
By the mixed inequality, every term in the sum is bounded below by its
coefficient times $f\,dV$, and the coefficients sum to one.  Hence
$tu+(1-t)v\in\calP_f$.

If $u\in\calP_f$, $v\in\PSH(\Omega)\cap C(\bOmega)$, and $s\geq0$, then
\[
        (\ddc(u+s v))^n
        =
        (\ddc u)^n
        +
        \sum_{k=1}^n
        \binom{n}{k}
        s^k
        (\ddc u)^{n-k}\wedge(\ddc v)^k
        .
\]
The mixed terms in the sum are positive Bedford--Taylor measures, while
$(\ddc u)^n\geq f\,dV$.  Thus $u+s v\in\calP_f$.
\end{proof}

\begin{lemma}[Complex strict subsolutions]\label{lem:strict-complex}
If $\Omega$ is bounded and $f\in C(\bOmega)$ is nonnegative, then for
every $\varphi\in C(\partial\Omega)$ there exists $u_0\in\calP_f$ with
$u_0|_{\partial\Omega}<\varphi$.
\end{lemma}

\begin{proof}
This is a standard elementary quadratic strict subsolution; compare
the subsolution construction in \cite[Section~4]{BedfordTaylor:76}.  We
include the short argument for convenience of the reader.
Choose $R>0$ with $\Omega\Subset B(0,R)$.  Let $c_n>0$ be defined by
$(\ddc |z|^2)^n=c_n\,dV$, and choose $A>0$ so large that
$A^n c_n\geq\|f\|_{L^\infty(\Omega)}$.  Then
$q(z)=A(|z|^2-R^2)$ belongs to $\PSH(\Omega)\cap C(\bOmega)$ and
$(\ddc q)^n\geq f\,dV$.  Subtracting a sufficiently large constant from
$q$ gives the desired strict boundary inequality.
\end{proof}

\begin{theorem}[Complex inhomogeneous Jensen formula]\label{thm:psh}
Let $\Omega\subset\C^n$ be a bounded domain, $x\in\Omega$, and
$f\in C(\bOmega)$, $f\geq0$.  Then for every
$\varphi\in C(\partial\Omega)$,
\[
        P_f^{\PSH}\varphi(x)
        =
        \inf_{\mu\in J_{x,\PSH}^{\partial}}
        \left(
        \int_{\partial\Omega}\varphi\,d\mu
        -
        B_f^{\PSH}(x,\mu)
        \right).
\]
\end{theorem}

\begin{proof}
Lemma~\ref{lem:complex-structural} gives the convexity and stability
assumptions in Theorem~\ref{thm:abstract-edwards}, with
$E=\partial\Omega$ and
$\mathcal F=\PSH(\Omega)\cap C(\bOmega)$.  Lemma~\ref{lem:strict-complex}
gives an admissible competitor below the prescribed boundary data, so
the envelope is not $-\infty$.  If $u|_{\partial\Omega}\leq\varphi$,
then the boundary maximum principle gives
$u\leq\max_{\partial\Omega}\varphi$ on $\bOmega$, so the envelope is not
$+\infty$.  The Jensen class
$J_x^E(\mathcal F)$ is exactly $J_{x,\PSH}^{\partial}$, and the
deficit is $B_f^{\PSH}$.  The abstract theorem gives the formula.
\end{proof}

\begin{remark}
The structural input in Lemma~\ref{lem:complex-structural} is the only
place where the continuous complex theorem uses Bedford--Taylor mixed
products.  No compactness or closedness of the corresponding epigraph is
needed for the separation argument itself.
\end{remark}

\subsection*{A current representation of the complex deficit}

We now pass from the real stress formulation to its current-theoretic
counterpart.  The proposition below proves the smooth estimate; the
existence of optimal fillings and the adjoint equation discussed
afterward remain formal unless additional regularity and compactness
hypotheses are imposed.

Fix $x\in\Omega$ and $\mu\in J_{x,\PSH}^{\partial}$.  A natural replacement
for a positive matrix-valued stress is a positive current $T$, supported
in $\bOmega$, of bidimension $(1,1)$, such that
\begin{equation}\label{eq:complex-jensen-current}
        \int_{\partial\Omega}\phi\,d\mu-\phi(x)
        =
        \langle T,\ddc\phi\rangle
\end{equation}
for smooth test functions $\phi$ on a neighbourhood of $\bOmega$.  The
shorthand
\[
        \ddc T=\mu-\delta_x
\]
means precisely \eqref{eq:complex-jensen-current}; the boundary measure
is part of the pairing on a neighbourhood of the closure.  Positivity of
$T$ is precisely the current-theoretic reflection of the Jensen inequality:
if $\phi$ is plurisubharmonic, then $\ddc\phi\geq0$ and hence
\[
        \phi(x)\leq\int_{\partial\Omega}\phi\,d\mu .
\]
Conversely, the Duval--Sibony--Harvey--Lawson bipolar argument gives
such fillings in the present setting.  Let $K=\bOmega$ and consider the
cone
\[
        \{\ddc S:\ S\geq0,\ \supp S\subset K\}.
\]
This cone is closed, since pairing with a fixed strictly
plurisubharmonic function controls the masses of the positive currents.
Its polar consists of smooth test functions with $\ddc\phi\geq0$ on
$K$, hence with plurisubharmonic restriction to $\Omega$.  Since
$\mu\in J_{x,\PSH}^{\partial}$, the bipolar theorem gives
$\mu-\delta_x=\ddc T$ for some positive current $T$ supported in
$\bOmega$; compare Duval--Sibony~\cite{DuvalSibony:95} and
Harvey--Lawson~\cite[Corollary~3.12]{HarveyLawsonCurrents:09}.  The
additional determinant cost below is the inhomogeneous feature.

\begin{proposition}[Complex Jensen currents and lower bound]
\label{prop:complex-current-lower-bound}
Let $T$ be a positive current satisfying
\eqref{eq:complex-jensen-current}.  Fix once and for all the contracted
basis identifying positive $(n-1,n-1)$-forms with positive Hermitian
cofactor matrices, and write the absolutely continuous part as
$T_{ac}=\Theta\,dV$; see, for instance, \cite[Chapter~3]{Klimek:91}.
Then there is a positive constant $c_{\C}=c_{\C}(\ddc,dV)$, depending
only on these conventions, such that, for every smooth plurisubharmonic
$u$ with $(\ddc u)^n\geq f\,dV$,
\begin{equation}\label{eq:complex-stress-lower-bound}
        \int_{\partial\Omega}u\,d\mu-u(x)
        \geq
        c_{\C}\int_{\Omega} f^{1/n}\det(\Theta)^{1/n}\,dV .
\end{equation}
\end{proposition}

\begin{proof}
For any such current,
\[
        \int_{\partial\Omega}u\,d\mu-u(x)
        =
        \langle T,\ddc u\rangle .
\]
The singular part of $T$ gives a nonnegative contribution, so
\[
        \langle T,\ddc u\rangle
        \geq
        \int_{\Omega}\ddc u\wedge T_{ac}.
\]
In the fixed contracted-basis convention, the pointwise matrix AM--GM
inequality for positive Hermitian matrices gives
\begin{equation}\label{eq:complex-stress-amgm}
        \ddc u\wedge T_{ac}
        \geq
        c_{\C}
        \left(\frac{(\ddc u)^n}{dV}\right)^{1/n}
        \det(\Theta)^{1/n}\,dV,
\end{equation}
which is the same matrix AM--GM inequality used in
Proposition~\ref{prop:real-stress-lower-bound}: if $H$ is the Hermitian
Hessian matrix of $u$, we apply the scalar AM--GM to the eigenvalues of
$\Theta^{1/2}H\Theta^{1/2}$, using the Hermitian square root; see again
\cite[Chapter~IV]{Bhatia:97}.  Since
$(\ddc u)^n\geq f\,dV$, integration gives
\eqref{eq:complex-stress-lower-bound}.
\end{proof}

This construction suggests the relaxed complex deficit
\begin{equation}\label{eq:complex-relaxed-deficit}
        \mathcal D_f^{\PSH}(x,\mu)
        =
        \sup_T
        c_{\C}\int_{\Omega} f^{1/n}\det(\Theta_T)^{1/n}\,dV,
\end{equation}
where the supremum is over positive Jensen currents satisfying
\eqref{eq:complex-jensen-current}, and $T_{ac}=\Theta_T\,dV$.  The singular part
of $T$ contributes to the ordinary Jensen gap but not to the full-dimensional
determinant in \eqref{eq:complex-relaxed-deficit}.  Thus the inhomogeneous
Monge--Amp\`ere lower bound is detected by the absolutely continuous part of the
current.  In the smooth or approximable setting covered by the proposition,
$\mathcal D_f^{\PSH}(x,\mu)\leq B_f^{\PSH}(x,\mu)$. Again, we leave the question
of whether equality holds under weaker natural hypotheses open; it would amount to a
nonlinear adjoint identity for the complex deficit.

\begin{remark}[Formal equality case]
The formal equality case gives the expected adjoint equation.  If a
smooth extremal $u$ and a maximizing current $T$ exist, and if no
relaxation gap occurs, equality in \eqref{eq:complex-stress-amgm}
would force the absolutely continuous part of $T$ to be proportional to
$(\ddc u)^{n-1}$:
\[
        T_{ac}=a\,(\ddc u)^{n-1}
\]
after absorbing normalization constants into $a$.  The current equation
then becomes, formally,
\[
        \ddc\bigl(a\,(\ddc u)^{n-1}\bigr)
        =
        \mu-\delta_x,
\]
together with
\[
        (\ddc u)^n=f\,dV.
\]
In this picture the deficit is the potential of $f$ against the
nonnegative adjoint Green density $a$ for the linearized complex
Monge--Amp\`ere operator at the extremal plurisubharmonic function.
Section~\ref{sec:linearized-harmonic-measure} proves this picture in
the smooth strictly elliptic case, with the Green kernel of $L_u$
playing the role of $a$.
\end{remark}

\subsection*{Upper bounded and bounded competitors}

There is also a trace version in which the competitors are not assumed
to be continuous on $\bOmega$.  If $u$ is upper bounded and
plurisubharmonic on $\Omega$, write
\[
        u^*(\zeta)
        =
        \limsup_{\Omega\ni z\to \zeta}u(z),
        \qquad
        \zeta\in\bOmega,
\]
for its upper semicontinuous regularization on the closure.  Suppose
that $\mathcal G$ is a cone of upper bounded plurisubharmonic functions
whose boundary regularizations are bounded upper semicontinuous on
$\partial\Omega$, and that $\calQ$ is a nonempty class of upper bounded
plurisubharmonic functions with bounded upper semicontinuous boundary
traces.  Assume that the pair $(\calQ,\mathcal G)$ satisfies
Assumption~\ref{ass:abstract} with the trace $u_E=u^*|_{\partial\Omega}$.
Define
\[
        P_{\calQ}\varphi(x)
        =
        \sup\left\{
        u(x):
        u\in\calQ,\
        u^*|_{\partial\Omega}\leq\varphi
        \right\}.
\]
For a boundary measure $\mu$ define
\[
        B_{\calQ}(x,\mu)
        =
        \inf_{u\in\calQ}
        \left(
        \int_{\partial\Omega}u^*\,d\mu-u(x)
        \right).
\]
The corresponding boundary Jensen class is
\[
        J_x^\partial(\mathcal G)
        =
        \left\{
        \mu\in\calP(\partial\Omega):
        v(x)\leq\int_{\partial\Omega}v^*\,d\mu
        \text{ for all }v\in\mathcal G
        \right\}.
\]

\begin{proposition}[Trace duality]\label{prop:psh-trace-duality}
With the hypotheses just stated, if
$P_{\calQ}\varphi(x)$ is finite, then
\[
        P_{\calQ}\varphi(x)
        =
        \inf_{\mu\in J_x^\partial(\mathcal G)}
        \left(
        \int_{\partial\Omega}\varphi\,d\mu
        -
        B_{\calQ}(x,\mu)
        \right),
\]
\end{proposition}

\begin{proof}
This is Theorem~\ref{thm:abstract-edwards} with
$E=\partial\Omega$ and the trace $u_E=u^*|_{\partial\Omega}$.  The
trace compatibilities are the elementary inequalities
\[
        (t u+(1-t)v)^*
        \leq
        t u^*+(1-t)v^*,
        \qquad
        (u+a)^*=u^*+a,
\]
restricted to $\partial\Omega$.  In this
notation the separating set is
\[
        \calK_{\calQ}
        =
        \left\{
        (g,u(x)-s):
        u\in\calQ,\
        g\in C(\partial\Omega),\
        g\geq u^*|_{\partial\Omega},\
        s\geq0
        \right\}.
\]
The bounded upper semicontinuity of the traces ensures that the
infimum over continuous majorants $g$ is
$\int_{\partial\Omega}u^*\,d\mu$, by the continuous-majorant formula
used in the proof of the abstract theorem.
\end{proof}

Taking $\mathcal G$ to be a suitable cone of upper bounded
plurisubharmonic functions and taking $\calQ=\calP_f^{ub}$, where
$\calP_f^{ub}$ denotes an upper bounded Monge--Amp\`ere subsolution
class with the trace and structural properties required above, gives
the upper bounded formula.  In the rest of the paper we use only the
following bounded Bedford--Taylor instance.

For the Bedford--Taylor application we shall use the bounded subclass
\[
        \calP_f^b
        =
        \left\{
        u\in\PSH(\Omega)\cap L^\infty(\Omega):
        (\ddc u)^n\geq f\,dV
        \right\}.
\]
Define $P_f^b$, $B_f^b$, and $J_{x,\PSH}^{b,\partial}$ by
\[
        P_f^b\varphi(x)
        =
        \sup\left\{
        u(x):
        u\in\calP_f^b,\
        u^*|_{\partial\Omega}\leq\varphi
        \right\},
\]
\[
        B_f^b(x,\mu)
        =
        \inf_{u\in\calP_f^b}
        \left(
        \int_{\partial\Omega}u^*\,d\mu-u(x)
        \right),
\]
and
\[
        J_{x,\PSH}^{b,\partial}
        =
        \left\{
        \mu\in\calP(\partial\Omega):
        v(x)\leq\int_{\partial\Omega}v^*\,d\mu
        \text{ for all }v\in\PSH(\Omega)\cap L^\infty(\Omega)
        \right\}.
\]

\begin{proposition}[Bounded Bedford--Taylor dual formula]
\label{prop:psh-bounded-duality}
Let $\Omega\subset\C^n$ be bounded, let $x\in\Omega$, and let
$f\in C(\bOmega)$, $f\geq0$.  Then for every
$\varphi\in C(\partial\Omega)$,
\[
        P_f^b\varphi(x)
        =
        \inf_{\mu\in J_{x,\PSH}^{b,\partial}}
        \left(
        \int_{\partial\Omega}\varphi\,d\mu
        -
        B_f^b(x,\mu)
        \right).
\]
\end{proposition}

\begin{proof}
The Bedford--Taylor products of bounded plurisubharmonic functions are
well-defined, positive, local, and continuous under decreasing locally
bounded limits \cite{BedfordTaylor:82}.  Convexity of $\calP_f^b$ follows
from Dinew's mixed Monge--Amp\`ere inequality
\cite[Theorem~1.3]{Dinew:09}, in the precise form stated in the proof of
Lemma~\ref{lem:complex-structural}.  Stability under addition of bounded
plurisubharmonic functions follows from positivity of mixed products.
Constants are harmless, and Lemma~\ref{lem:strict-complex} gives a
bounded admissible competitor below the prescribed boundary data.  Thus
Proposition~\ref{prop:psh-trace-duality} applies with
$\mathcal G=\PSH(\Omega)\cap L^\infty(\Omega)$ and
$\calQ=\calP_f^b$.  The maximum principle gives finiteness of the
envelope.
\end{proof}

\begin{proposition}[Standard Bedford--Taylor stability and regularization]
\label{lem:bt-stability-regularization}
Let $\Omega\subset\C^n$ be bounded and let
$f\in C(\bOmega)$, $f\geq0$.
\begin{enumerate}[label=\textup{(\roman*)}]
    \item If bounded plurisubharmonic functions $v$ and $w$ satisfy
    $(\ddc v)^n\geq f\,dV$ and $(\ddc w)^n\geq f\,dV$, then
    $\max\{v,w\}$ satisfies the same inequality.
    \item If $v_j\downarrow v$ is a decreasing locally bounded sequence
    of plurisubharmonic functions and
    $(\ddc v_j)^n\geq f\,dV$ for all $j$, then
    $(\ddc v)^n\geq f\,dV$.
    \item Let $u\in\PSH(\Omega)\cap L^\infty(\Omega)$, and assume
    $(\ddc u)^n\geq f\,dV$.  Let $\chi_r$ be a monotone radial
    smoothing kernel, so that $u_r=u*\chi_r$ is smooth
    plurisubharmonic on $\Omega_r$ and decreases to $u$ on compact
    subsets as $r\downarrow0$.
    For every $a>0$ there is $r_0>0$ such that
    \[
            \bigl(\ddc(u_r+a|z|^2)\bigr)^n\geq f\,dV
            \quad\text{on }\Omega_{2r}
    \]
    whenever $0<r<r_0$, where
    $\Omega_{2r}=\{z:\operatorname{dist}(z,\partial\Omega)>2r\}$.
\end{enumerate}
\end{proposition}

\begin{proof}
The first assertion is the Bedford--Taylor maximum principle for
locally bounded plurisubharmonic functions:
\[
        (\ddc\max\{v,w\})^n
        \geq
        \mathbf 1_{\{v\geq w\}}(\ddc v)^n
        +
        \mathbf 1_{\{v<w\}}(\ddc w)^n .
\]
The second assertion is the Bedford--Taylor monotone convergence
theorem for locally bounded decreasing sequences.  See
\cite[Theorem~2.4]{BedfordTaylor:82} and the Bedford--Taylor maximum
principle from \cite{BedfordTaylor:76,BedfordTaylor:82}.

For the regularization statement, put $F=f^{1/n}$.  The main theorem of
Guedj--Lu--Zeriahi \cite[Main theorem]{GuedjLuZeriahi:19} gives, for standard
regularizations on $\Omega_{2r}$,
\[
        (\ddc u_r)^n\geq (F*\chi_r)^n\,dV .
\]
Equivalently,
\[
        \det_{dV}(\ddc u_r)^{1/n}
        \geq F*\chi_r ,
\]
where $\det_{dV}$ denotes the determinant relative to the fixed volume
form.  Let
\[
        c_{\ddc,dV}
        =
        \det_{dV}(\ddc |z|^2)^{1/n}>0 .
\]
The Hermitian Minkowski determinant inequality, equivalently the
concavity and homogeneity of $\det_{dV}^{1/n}$ on positive Hermitian
matrices, gives
\[
        \det_{dV}\bigl(\ddc(u_r+a|z|^2)\bigr)^{1/n}
        \geq
        F*\chi_r+c_{\ddc,dV}a .
\]
Since $F$ is uniformly continuous on $\bOmega$, the last quantity is at
least $F$ on $\Omega_{2r}$ for all sufficiently small $r$.  This proves
the local regularization assertion.
\end{proof}

Next, we prove a global approximation theorem for bounded plurisubharmonic
functions with control of the Monge--Amp\`ere measures. More precisely,

\begin{theorem}[Bounded inhomogeneous approximation]\label{thm:psh-inhom-approx}
Let $\Omega\subset\C^n$ be a bounded B-regular domain and let
$f\in C(\bOmega)$, $f\geq0$.  Let
$u\in\PSH(\Omega)\cap L^\infty(\Omega)$, and assume that
\[
        (\ddc u)^n\geq f\,dV
\]
in the Bedford--Taylor sense.  Then there exists a decreasing sequence
$u_j\in\PSH(\Omega)\cap C(\bOmega)$ such that
\[
        (\ddc u_j)^n\geq f\,dV
        \quad\text{and}\quad
        u_j\searrow u^*
        \quad\text{on }\bOmega.
\]
\end{theorem}

\begin{proof}
We use the stability and local regularization facts collected in
Proposition~\ref{lem:bt-stability-regularization}.  The argument is based on the
monotone approximation construction of Wikström~\cite[Theorem~4.1]{Wikstrom:01},
with an additional quadratic correction in the
regularization step to preserve the lower density.

\emph{Step 1: Barriers and local regularizations.}
Choose a negative plurisubharmonic exhaustion
$\rho\in\PSH(\Omega)\cap C(\bOmega)$, with
$\rho=0$ on $\partial\Omega$; such an
exhaustion is available on bounded B-regular domains
\cite{Sibony:87,Demailly:87,Blocki:96}.  Choose
$R>0$ with $\Omega\Subset B(0,R)$ and choose $A>0$ so large that
\[
        q(z)=A(|z|^2-R^2)
\]
satisfies $q\leq0$ on $\bOmega$ and $(\ddc q)^n\geq f\,dV$.
Put
\[
        \psi=u^*|_{\partial\Omega}-q|_{\partial\Omega}.
\]
Choose $\psi_m\in C(\partial\Omega)$ with
$\psi_m>\psi$, $\psi_{m+1}\leq\psi_m$, and
$\psi_m\downarrow\psi$ on $\partial\Omega$.  By B-regularity, let
$h_m\in\PSH(\Omega)\cap C(\bOmega)$ be the maximal
plurisubharmonic function with boundary values $\psi_m$.  The
comparison principle gives $h_{m+1}\leq h_m$.

Set
\[
        \eta_m
        =
        \min_{\partial\Omega}
        \bigl(q+\psi_m-u^*\bigr)
        >0.
\]
Choose decreasing positive numbers $\delta_m\downarrow0$ and
$a_m\downarrow0$ so small that
\[
        \delta_m<\frac{\eta_m}{16},
        \qquad
        a_mR^2<\frac{\eta_m}{16}.
\]
Since $u^*-(q+h_m)$ is upper semicontinuous and is at most
$-\eta_m$ on $\partial\Omega$, there is a boundary collar $V_m$ such
that
\[
        u\leq q+h_m-\frac{3\eta_m}{4}
        \quad\text{on }V_m\cap\Omega.
\]
Shrinking $V_m$ if necessary, we may also arrange that
$m\rho\geq-\eta_m/8$ on $V_m$.

We now choose $r_m\downarrow0$ recursively.  The radius $r_m$ is chosen
small enough that Proposition~\ref{lem:bt-stability-regularization}
applies to
\[
        A_m=u_{r_m}+a_m|z|^2
\]
on $\Omega_{2r_m}:=\{z:\operatorname{dist}(z,\partial\Omega)>2r_m\}$,
small enough that
$\{z:\operatorname{dist}(z,\partial\Omega)\leq 4r_m\}\subset V_m$, and
small enough that the oscillation of $q+h_m$ on every ball of radius
$r_m$ contained in $V_m$ is less than $\eta_m/8$.  We also require
$r_m\leq r_{m-1}$ when $m>1$.

For later use we record the quantitative collar estimate.  Put
\[
        W_m
        =
        \{z\in\Omega:2r_m<\operatorname{dist}(z,\partial\Omega)<3r_m\}
        \subset V_m
\]
and
\[
        b_m=q+h_m+m\rho .
\]
If $z\in W_m$, then the convolution defining $u_{r_m}(z)$ only uses
points of $V_m$.  Hence
\[
        u_{r_m}(z)
        \leq
        q(z)+h_m(z)-\frac{5\eta_m}{8}.
\]
Together with $a_mR^2<\eta_m/16$, this gives
\[
        A_m(z)
        \leq
        q(z)+h_m(z)-\frac{9\eta_m}{16}
        \leq
        q(z)+h_m(z)-\frac{\eta_m}{2}.
\]
Since also $m\rho\geq-\eta_m/8$ on $W_m$, we obtain
\[
        b_m(z)-A_m(z)
        \geq
        \frac{3\eta_m}{8}
        \quad (z\in W_m).
\]

\emph{Step 2: Glued continuous subsolutions.}
Define
\[
        \widetilde u_m
        =
        \begin{cases}
        \max\{A_m-\delta_m,b_m\},
            & z\in\Omega_{2r_m},\\
        b_m,
            & z\in\Omega\setminus\Omega_{2r_m}.
        \end{cases}
\]
The collar estimate gives
\[
        b_m-(A_m-\delta_m)
        \geq
        \frac{3\eta_m}{8}+\delta_m>0
        \quad\text{on }W_m .
\]
Hence $\widetilde u_m=b_m$ on the open collar
$\{z\in\Omega:\operatorname{dist}(z,\partial\Omega)<3r_m\}$.  The
standard gluing lemma for plurisubharmonic functions
\cite[Proposition~2.9.26]{Klimek:91} shows that
$\widetilde u_m\in\PSH(\Omega)\cap C(\bOmega)$.  Moreover
$(\ddc A_m)^n\geq f\,dV$ by construction, while
\[
        (\ddc b_m)^n
        =
        \bigl(\ddc q+\ddc(h_m+m\rho)\bigr)^n
        \geq
        (\ddc q)^n
        \geq f\,dV,
\]
because all mixed Bedford--Taylor products are positive.  Hence the
maximum in the definition of $\widetilde u_m$ has the desired lower
bound on $\Omega_{2r_m}$ by
Proposition~\ref{lem:bt-stability-regularization}.  On the open collar
$\{\operatorname{dist}(\cdot,\partial\Omega)<3r_m\}$ the function is
$b_m$, so the same bound holds there.  These two open sets cover
$\Omega$, and Bedford--Taylor locality gives
$(\ddc\widetilde u_m)^n\geq f\,dV$ on all of $\Omega$.

The sequence $\widetilde u_m$ has the correct pointwise limit.  On
$\partial\Omega$ it equals $q+\psi_m$, and this decreases to $u^*$.
The boundary data $\psi_m$ are uniformly bounded, and hence so are the
maximal extensions $h_m$.  Therefore, for each fixed interior point,
$m\rho\to-\infty$, while $A_m-\delta_m\to u$; hence
$\widetilde u_m\to u$ in $\Omega$.

\emph{Step 3: decreasing rearrangement.}
It remains only to make the sequence decreasing without losing
continuity or the Monge--Amp\`ere inequality.  The role of the auxiliary
functions $C_K$ is to dominate the tail
$\{\widetilde u_m:m\geq K\}$ while still converging to $u^*$.  For
$K\geq1$ put
\[
        C_K
        =
        \begin{cases}
        \max\{A_K,b_K+\delta_K\},
            & z\in\Omega_{2r_K},\\
        b_K+\delta_K,
            & z\in\Omega\setminus\Omega_{2r_K}.
        \end{cases}
\]
The collar estimate gives
\[
        (b_K+\delta_K)-A_K
        \geq
        \frac{3\eta_K}{8}+\delta_K>0
        \quad\text{on }W_K .
\]
Thus $C_K=b_K+\delta_K$ on
$\{\operatorname{dist}(\cdot,\partial\Omega)<3r_K\}$, and the same
gluing and locality argument shows that
$C_K\in\PSH(\Omega)\cap C(\bOmega)$ and
$(\ddc C_K)^n\geq f\,dV$.

We claim that, for $m\geq K$,
\[
        \widetilde u_m\leq C_K.
\]
If $z\in\Omega_{2r_K}$, then $A_m$ is defined at $z$ and the monotonic
choices give
\[
        A_m\leq A_K,
        \qquad
        b_m\leq b_K
\]
because $u_{r_m}\leq u_{r_K}$, $a_m\leq a_K$, $h_m\leq h_K$, and
$m\rho\leq K\rho$.  Hence
$\widetilde u_m(z)\leq C_K(z)$.

It remains to consider the boundary collar
$\Omega\setminus\Omega_{2r_K}$.  If $z\notin\Omega_{2r_m}$, then
$\widetilde u_m(z)=b_m(z)\leq b_K(z)\leq C_K(z)$.  If
$z\in\Omega_{2r_m}$, then every point used by the convolution defining
$u_{r_m}(z)$ lies in $V_K$, since $r_m\leq r_K$ and
$\operatorname{dist}(z,\partial\Omega)\leq2r_K$.  On this ball,
$u\leq q+h_K-3\eta_K/4$, and the oscillation of $q+h_K$ is less than
$\eta_K/8$.  Therefore
\[
        A_m(z)
        \leq
        q(z)+h_K(z)-\frac{3\eta_K}{4}
        +\frac{\eta_K}{8}
        +a_mR^2
        \leq
        q(z)+h_K(z)-\frac{9\eta_K}{16}.
\]
Since $K\rho\geq-\eta_K/8$ on $V_K$ and
$\delta_K>0$, this implies
\[
        A_m(z)-\delta_m
        \leq
        q(z)+h_K(z)-\frac{9\eta_K}{16}
        \leq
        q(z)+h_K(z)-\frac{\eta_K}{8}
        \leq
        b_K(z)
        \leq
        b_K(z)+\delta_K.
\]
Together with $b_m\leq b_K$, this proves the claim in the collar.
The same estimates also give the stronger comparison needed for the
ceiling functions: if $m>K$, then $A_m\leq C_K$ on
$\Omega_{2r_m}$, and
\[
        b_m+\delta_m\leq b_K+\delta_K\leq C_K.
\]
Consequently $C_m\leq C_K$ for every $m>K$.

For each fixed point of $\Omega$, the term $b_K+\delta_K$ tends to
$-\infty$ because $K\rho\to-\infty$, whereas $A_K\to u$.  On the
boundary, $C_K=q+\psi_K+\delta_K$, which decreases to
$u^*|_{\partial\Omega}$.  Thus
\[
        C_K\to u^*
        \quad\text{pointwise on }\bOmega .
\]

For $K\geq j$ define
\[
        U_{j,K}
        =
        \max\{\widetilde u_j,\ldots,\widetilde u_K,C_K\}.
\]
Then $U_{j,K}\in\PSH(\Omega)\cap C(\bOmega)$,
$(\ddc U_{j,K})^n\geq f\,dV$.  The preceding comparison gives
$\widetilde u_{K+1}\leq C_K$, and the strengthened comparison gives
$C_{K+1}\leq C_K$; hence $U_{j,K}$ decreases as $K\to\infty$.  Let
\[
        u_j=\sup_{m\geq j}\widetilde u_m .
\]
For each fixed point,
$\max\{\widetilde u_j,\ldots,\widetilde u_K\}$ increases to $u_j$, while
$C_K\to u^*$.  Since $\widetilde u_m\to u^*$ pointwise, one has
$u_j\geq u^*$, and therefore the pointwise limit of $U_{j,K}$ is
$\max\{u_j,u^*\}=u_j$.  Thus $U_{j,K}\downarrow u_j$.
The functions $u_j$ are lower semicontinuous as suprema of continuous
functions, while the decreasing representation by $U_{j,K}$ shows that
they are upper semicontinuous.  Hence $u_j$ is continuous on $\bOmega$.
Since they are decreasing limits of the plurisubharmonic functions
$U_{j,K}$, we get
$u_j\in\PSH(\Omega)\cap C(\bOmega)$, and monotone convergence gives
$(\ddc u_j)^n\geq f\,dV$.  Finally,
$u_j\downarrow u^*$ on $\bOmega$, since
$\widetilde u_m\to u^*$ pointwise on $\bOmega$.
\end{proof}

As a corollary of the approximation theorem and the dual formulas, we can give
an alternative proof of the continuity of the Bedford--Taylor solution $P_f^b
\varphi$ to the Dirichlet problem with continuous boundary data and measures
with continuous density. This result is well-known~\cite{BedfordTaylor:82}, but
the duality approach provides a new perspective on the problem.

\begin{corollary}[Bedford--Taylor continuity from duality]\label{cor:bt-continuity}
Let $\Omega\subset\C^n$ be a bounded B-regular domain and let
$f\in C(\bOmega)$, $f\geq0$.  For $\varphi\in C(\partial\Omega)$, let
$P_f^c\varphi$ denote the envelope obtained from
$\PSH(\Omega)\cap C(\bOmega)$, and let $P_f^b\varphi$ denote the
bounded Bedford--Taylor envelope above.  Then
\[
        P_f^c\varphi=P_f^b\varphi
        \quad\text{in }\Omega.
\]
Consequently $P_f^c\varphi$ is continuous in $\Omega$.
\end{corollary}

\begin{proof}
First note that the bounded envelope is already upper semicontinuous.
Let $U=P_f^b\varphi$ and let $U^*$ be its upper semicontinuous
regularization in $\Omega$.  The strict subsolution lemma gives a
bounded admissible competitor, and the boundary maximum principle gives
a uniform upper bound, so $U$ is finite.  Since $\calP_f^b$ is
stable under maxima, Choquet's lemma and Bedford--Taylor continuity for
increasing locally bounded sequences \cite[Proposition~5.2]{BedfordTaylor:82}
show that
\[
        U^*\in\PSH(\Omega)\cap L^\infty(\Omega),
        \qquad
        (\ddc U^*)^n\geq f\,dV .
\]
Let $h\in\PSH(\Omega)\cap C(\bOmega)$ be the maximal
plurisubharmonic function with boundary values $\varphi$, which exists
by B-regularity.  Every bounded competitor is dominated by $h$, and
hence $U^*\leq h$ in $\Omega$.  Thus the boundary trace of $U^*$ is at
most $\varphi$, so $U^*$ is itself an admissible bounded competitor.
It follows that $U^*\leq U$, while the reverse inequality is automatic.
Hence $P_f^b\varphi=U=U^*$.

The inequality $P_f^c\varphi\leq P_f^b\varphi$ is immediate, since the
continuous competitors form a subclass of the bounded competitors.

The reverse inequality follows from the dual formula.  On a B-regular
domain, Jensen measures for continuous plurisubharmonic functions and
for bounded plurisubharmonic functions coincide; this is the
approximation theorem behind the usual maximal case
\cite[Theorem~4.1 and Corollary~4.3]{Wikstrom:01}.  It remains only to
compare the deficits.  Since
$\calP_f\subset\calP_f^b$, one has
\[
        B_f^b(x,\mu)\leq B_f^{\PSH}(x,\mu).
\]
Conversely, let $u\in\calP_f^b$.  By
Theorem~\ref{thm:psh-inhom-approx}, there are
$u_j\in\PSH(\Omega)\cap C(\bOmega)$ with
$(\ddc u_j)^n\geq f\,dV$ and $u_j\searrow u^*$ on $\bOmega$.  Hence, by
monotone convergence,
\[
        B_f^{\PSH}(x,\mu)
        \leq
        \lim_{j\to\infty}
        \left(
        \int_{\partial\Omega}u_j\,d\mu-u_j(x)
        \right)
        =
        \int_{\partial\Omega}u^*\,d\mu-u(x).
\]
Taking the infimum over $u\in\calP_f^b$ gives
$B_f^{\PSH}(x,\mu)\leq B_f^b(x,\mu)$, and therefore the deficits agree.
The continuous dual formula, Theorem~\ref{thm:psh}, and the bounded
dual formula, Proposition~\ref{prop:psh-bounded-duality}, now give
$P_f^c\varphi=P_f^b\varphi$.

The common envelope is lower semicontinuous as a supremum of continuous
competitors and upper semicontinuous by the first paragraph.  This
proves interior continuity.  The final boundary statement follows from
the equality of the two envelopes and the corresponding
Bedford--Taylor properties of the bounded envelope.
\end{proof}

\section{Optimal Jensen Measures and Linearized Harmonic Measure}
\label{sec:linearized-harmonic-measure}

The preceding dual formulas identify the value of the envelope by
minimizing over Jensen measures.  In the smooth strictly elliptic case,
these minimizing measures have a concrete PDE interpretation: they are
the harmonic measures for the linearized Monge--Amp\`ere operator at the
solution.  This links the deficit formula to the linearized
Monge--Amp\`ere theory of Caffarelli--Guti{\'e}rrez
\cite{CaffarelliGutierrez:97}; the underlying smooth Dirichlet theory is
classical in the real and complex settings
\cite{CaffarelliKohnNirenbergSpruck:85,CaffarelliNirenbergSpruck:84}.

\begin{theorem}[Linearized harmonic measure, real case]
\label{thm:real-linearized-harmonic-measure}
Let $\Omega\subset\R^n$ be a bounded smooth uniformly convex domain,
let $f\in C^\infty(\bOmega)$ be positive, and let
$u\in C^\infty(\bOmega)$ be strictly convex with
\[
        \det D^2u=f
        \quad\text{in }\Omega,
        \qquad
        u|_{\partial\Omega}=\varphi .
\]
Set
\[
        L_u h=\operatorname{cof}(D^2u):D^2h,
\]
and let $\omega_u^x$ be the harmonic measure for $L_u$ at
$x\in\Omega$, i.e.
\[
        h(x)=\int_{\partial\Omega}g\,d\omega_u^x
\]
whenever $L_u h=0$ in $\Omega$ and $h|_{\partial\Omega}=g$.  Then
$\omega_u^x\in J_x^\partial$ and
\[
        B_f(x,\omega_u^x)
        =
        \int_{\partial\Omega}\varphi\,d\omega_u^x-u(x).
\]
Consequently $u(x)=P_f\varphi(x)$ and $\omega_u^x$ attains the infimum
in Theorem~\ref{thm:real}.
\end{theorem}

\begin{proof}
Set $a^{ij}=\operatorname{cof}(D^2u)_{ij}$.  The matrix $a^{ij}$ is
smooth and uniformly elliptic.  We shall use the following weak
comparison fact: if $\eta\in C(\bOmega)$ and
$a^{ij}D_{ij}\eta$ is a nonnegative distribution, then $\eta$ is bounded
above by the $L_u$-harmonic extension of its boundary values.  This is
the standard maximum principle for distributional subsolutions of a
smooth uniformly elliptic operator; one may prove it by interior
mollification and boundary barriers, or quote the weak maximum principle
in \cite[Chapter~8]{GilbargTrudinger:01}.

If $v\in C(\bOmega)\cap\Conv(\Omega)$, then $D^2v$ is a positive
semidefinite matrix-valued measure.  Since $a^{ij}$ is positive
definite, $L_u v\geq0$ in the sense of distributions.  Comparing $v$
with its $L_u$-harmonic extension gives
\[
        v(x)\leq\int_{\partial\Omega}v\,d\omega_u^x,
\]
so $\omega_u^x$ is a boundary Jensen measure.

Let $w\in\calA_f$.  Since $\operatorname{cof}(D^2u)$ is smooth,
$\operatorname{cof}(D^2u):D^2w$ is a Radon measure.  Writing the
absolutely continuous part of $D^2w$ as $A\,dx$, the inequality
$\MA_{\R}(w)\geq f\,dx$ gives $\det A\geq f$ a.e.  The singular part of
$D^2w$ is positive semidefinite.  For the absolutely continuous part,
the matrix AM--GM inequality and
$\det\operatorname{cof}(D^2u)=(\det D^2u)^{n-1}$ give
\[
        \operatorname{cof}(D^2u):A
        \geq
        n(\det D^2u)^{(n-1)/n}(\det A)^{1/n}
        \geq
        n f .
\]
Therefore
\[
        L_u w
        =
        \operatorname{cof}(D^2u):D^2w
        \geq
        n f
        =
        L_u u
\]
as an inequality of measures.  Hence $L_u(w-u)\geq0$ in the weak sense,
and the same comparison principle applied to $w-u$ yields
\[
        w(x)-u(x)
        \leq
        \int_{\partial\Omega}(w-\varphi)\,d\omega_u^x .
\]
Equivalently,
\[
        \int_{\partial\Omega}w\,d\omega_u^x-w(x)
        \geq
        \int_{\partial\Omega}\varphi\,d\omega_u^x-u(x).
\]
Taking the infimum over $w\in\calA_f$ gives one inequality for the
deficit, while the choice $w=u$ gives equality.  If in addition
$w|_{\partial\Omega}\leq\varphi$, the preceding comparison gives
$w\leq u$, so $P_f\varphi=u$.  The final assertion follows from
Theorem~\ref{thm:real}.
\end{proof}

\begin{theorem}[Linearized harmonic measure, complex case]
\label{thm:complex-linearized-harmonic-measure}
Let $\Omega\subset\C^n$ be a bounded strongly pseudoconvex domain with
smooth boundary, let $f\in C^\infty(\bOmega)$ be positive, and let
$u\in C^\infty(\bOmega)$ be strictly plurisubharmonic with
\[
        (\ddc u)^n=f\,dV
        \quad\text{in }\Omega,
        \qquad
        u|_{\partial\Omega}=\varphi .
\]
Define the linearized complex Monge--Amp\`ere operator by
\[
        L_u h\,dV
        =
        n\,\ddc h\wedge(\ddc u)^{n-1},
\]
and let $\omega_u^x$ be the harmonic measure for $L_u$ at $x$.  Then
$\omega_u^x\in J_{x,\PSH}^{\partial}$ and
\[
        B_f^{\PSH}(x,\omega_u^x)
        =
        \int_{\partial\Omega}\varphi\,d\omega_u^x-u(x).
\]
Consequently $u(x)=P_f^{\PSH}\varphi(x)$ and $\omega_u^x$ attains the
infimum in Theorem~\ref{thm:psh}.
\end{theorem}

\begin{proof}
The operator $L_u$ has smooth uniformly elliptic coefficients.  As in
the real case, we use the weak maximum principle for continuous
functions whose image under $L_u$ is a nonnegative measure; it follows
from the smooth maximum principle by local regularization and a barrier
argument at the boundary, or from the weak maximum principle
\cite[Chapter~8]{GilbargTrudinger:01}.

If $v\in\PSH(\Omega)\cap C(\bOmega)$, then $\ddc v\geq0$ as a positive
current, and hence $L_u v\geq0$ as a measure.  Comparing $v$ with the
$L_u$-harmonic extension of its boundary values gives
\[
        v(x)\leq\int_{\partial\Omega}v\,d\omega_u^x,
\]
so $\omega_u^x$ is a plurisubharmonic Jensen measure.

Let $w\in\calP_f$.  The mixed Monge--Amp\`ere inequality
\cite[Theorem~1.3]{Dinew:09}, applied to $w$ and $n-1$ copies of $u$,
gives
\[
        L_u w\,dV
        =
        n\,\ddc w\wedge(\ddc u)^{n-1}
        \geq
        n f\,dV
        =
        L_u u\,dV .
\]
Thus $L_u(w-u)\geq0$ as a measure, and the same comparison argument
gives
\[
        \int_{\partial\Omega}w\,d\omega_u^x-w(x)
        \geq
        \int_{\partial\Omega}\varphi\,d\omega_u^x-u(x).
\]
The choice $w=u$ gives equality in the deficit.  If
$w|_{\partial\Omega}\leq\varphi$, the comparison also gives $w\leq u$,
so $P_f^{\PSH}\varphi=u$.  The dual optimality follows from
Theorem~\ref{thm:psh}.
\end{proof}

\begin{corollary}[Boundary derivative of the solution operator]
\label{cor:boundary-derivative}
In either smooth setting of
Theorem~\ref{thm:real-linearized-harmonic-measure} or
Theorem~\ref{thm:complex-linearized-harmonic-measure}, suppose
$\varphi_t=\varphi+t\psi$ and that the corresponding solutions $u_t$
exist smoothly for $|t|$ small, with the same right-hand side $f$.  Then
\[
        \left.\frac{d}{dt}\right|_{t=0}u_t(x)
        =
        \int_{\partial\Omega}\psi\,d\omega_u^x .
\]
\end{corollary}

\begin{proof}
Differentiate the Monge--Amp\`ere equation at $t=0$.  In the real case
this gives
\[
        \operatorname{cof}(D^2u):D^2\dot u=0,
\]
and in the complex case it gives
\[
        n\,\ddc\dot u\wedge(\ddc u)^{n-1}=0.
\]
In both cases $\dot u|_{\partial\Omega}=\psi$, so the defining property
of the harmonic measure gives the formula.
\end{proof}

The stress and current pictures of
Propositions~\ref{prop:real-stress-lower-bound}
and~\ref{prop:complex-current-lower-bound} recover the same measures.
In the real case, let $G_u(x,\cdot)$ be the Green kernel
for $L_u$, normalized so that
\[
        \int_{\partial\Omega}h\,d\omega_u^x-h(x)
        =
        \int_{\Omega}G_u(x,y)L_u h(y)\,dy .
\]
Then
\[
        \Gamma_x
        =
        G_u(x,y)\operatorname{cof}(D^2u(y))\,dy
\]
is a Jensen stress for $(x,\omega_u^x)$, and
\[
        B_f(x,\omega_u^x)
        =
        n\int_{\Omega}G_u(x,y)f(y)\,dy .
\]
In the complex case, with the normalization
$L_u h\,dV=n\,\ddc h\wedge(\ddc u)^{n-1}$, the analogous positive
current is
\[
        T_x
        =
        n\,G_u(x,\cdot)(\ddc u)^{n-1}.
\]
Thus the deficit is the potential of the Monge--Amp\`ere density against
the Green kernel of the linearized equation.

\end{document}